\documentclass[12pt]{amsart}

\usepackage{amsmath}
\usepackage{amsfonts}
\usepackage{latexsym}
\usepackage{graphicx}
\usepackage{amssymb}
\usepackage{amsthm}
\usepackage[margin=3cm]{geometry}
\usepackage{url}
\usepackage{color}
\usepackage{enumerate}
\usepackage[shortlabels]{enumitem} 
\usepackage[small]{caption}
\usepackage{cite}       

\usepackage{todonotes}

\usepackage{hyperref}

\usepackage[T1]{fontenc}     
\usepackage{lmodern}         
\usepackage[utf8]{inputenc}  

\newtheorem{definition}{Definition}
\newtheorem{lemma}[definition]{Lemma}
\newtheorem{proposition}[definition]{Proposition}
\newtheorem{example}[definition]{Example}

\newtheorem{corollary}[definition]{Corollary}

\newtheorem{theorem}[definition]{Theorem}

\newcommand{\Z}{\mathbb{Z}}

\newcommand{\R}{\mathbb{R}}
\newcommand{\bzero}{\mathbf{0}}
\newcommand{\ba}{\mathbf{a}}
\newcommand{\bb}{\mathbf{b}}
\newcommand{\bt}{\mathbf{t}}
\newcommand{\bx}{\mathbf{x}}
\newcommand{\by}{\mathbf{y}}
\newcommand{\br}{\mathbf{r}}
\newcommand{\diag}{\mathrm{diag}}

\begin{document}

\title{A note on matrices mapping a positive vector onto its element-wise inverse}


\author[S.~Labb\'e]{S\'ebastien Labb\'e}
\address[S.~Labb\'e]{CNRS, LaBRI, UMR 5800, F-33400 Talence, France}
\email{sebastien.labbe@labri.fr}
\urladdr{http://www.slabbe.org/}




\date{\today}
\begin{abstract}
For any primitive matrix $M\in\R^{n\times n}$ 
with positive diagonal entries,
we prove the existence and uniqueness of a
positive vector $\bx=(x_1,\dots,x_n)^t$ such that
$M\bx=(\frac{1}{x_1},\dots,\frac{1}{x_n})^t$.
The contribution of this note is to provide an alternative proof of a result of
Brualdi et al. (1966) on the diagonal equivalence of a nonnegative matrix to a
stochastic matrix.
\end{abstract}

\maketitle

\section{Introduction}

In this note, we consider 
matrices mapping a vector with positive entries onto its element-wise inverse.
We prove unicity and existence of such a vector for primitive matrices,
that is nonnegative matrices some power of which is positive,
 with positive diagonal entries.
The main result is:

\begin{theorem}\label{thm:main}
Let $M\in\R^{n\times n}_{\geq0}$ be a primitive matrix
with positive diagonal entries.
Then there exists a unique
vector $\bx=(x_1,\dots,x_n)^t$ with positive entries such that
$M\bx=(\frac{1}{x_1},\dots,\frac{1}{x_n})^t$.
\end{theorem}

It turns out that this question was already answered in 1966 under an
equivalent form.
In \cite{brualdi_diagonal_1966}, it was proved that if $A$ is a nonnegative
square matrix with positive diagonal entries, then there exists a unique
diagonal matrix $D$ with positive diagonal entries such that $DAD$ is row
stochastic (see also \cite{sinkhorn_relationship_1966} who proved it for
positive matrices $A$).
The equivalence is explained below in Lemma~\ref{lem:3equivalent}.

As a consequence, the contribution of this note is to provide an alternative
proof of the above result. Unicity for primitive matrices is obtained as a
consequence of Perron theorem whereas existence for nonnegative matrices with
positive diagonal entries is deduced from the Brouwer fixed-point theorem. 


This note is structured as follows.
In Section 2, we present an equivalent system of quadratic equations to be solved.
In Section 3, we deduce unicity for primitive matrices from Perron theorem.
In Section 4, we reduce the question to finding fixed-points of a function
and we recall Brouwer and Banach fixed-point theorems in Section 5.
In Section 6, we use Brouwer fixed-point theorem to prove existence for
nonnegative matrices with positive diagonal entries.
In Section 7, we use Banach fixed-point theorem to prove existence and unicity
for nonnegative matrices with relatively large enough diagonal entries
including matrices which are not primitive (Proposition~\ref{prop:banach-unicity}).


\section{A system of quadratic equations}

We say that a vector or a matrix is \emph{nonnegative} (resp. \emph{positive})
if all of its entries are nonnegative (resp. positive).

\begin{lemma}\label{lem:3equivalent}
Let $M=(m_{ij})\in\R^{n\times n}$ be a nonnegative matrix and $\bx=(x_1,\dots,x_n)^t$ be a
positive vector. The following conditions are equivalent.
\begin{enumerate}[\rm (i)]
\item $M\bx=(\frac{1}{x_1},\dots,\frac{1}{x_n})^t$,
\item $\diag(\bx) M \diag(\bx)$ is a stochastic matrix,
\item for every $i\in\{1,\dots,n\}$,
\begin{equation}\label{eq:2}
x_i\sum_{j=1}^n m_{ij}x_j = 1.
\end{equation}
\end{enumerate}
\end{lemma}

\begin{proof}
    (i) $\iff$ (ii).
The matrix $\diag(\bx) M \diag(\bx)$ is stochastic if and only if
$(1,\dots, 1)^t$ is a right eigenvector with eigenvalue $1$, that is,
\begin{equation}\label{eq:stochastic}
\diag(\bx) M
\diag(\bx)
(1,\dots, 1)^t
= (1,\dots, 1)^t
\end{equation}
which is equivalent to $M\bx=\diag(\bx)^{-1}(1,\dots, 1)^t=(\frac{1}{x_1},\dots,\frac{1}{x_n})^t$

    (ii) $\iff$ (iii).
Let $\br_i$ be the $i$-th row of the matrix $M$.
We develop \eqref{eq:stochastic} and we get
\[
\diag(\bx) M \bx
=
\diag(\bx) (\br_1\cdot\bx,\dots,\br_n\cdot\bx)^t
=
(x_1\br_1\cdot\bx,\dots,x_n\br_n\cdot\bx)^t
=
(1,\dots, 1)^t.
\]
This equation is verified if and only if, for each $i\in\{1,\dots,n\}$, 
the quadratic Equation~\eqref{eq:2} in $x_1,\dots,x_n$ holds.
\end{proof}

The system of equations~\eqref{eq:2}
for $i\in\{1,\dots,n\}$ is illustrated in Figure~\ref{fig:n23}
for $n=2$ and $n=3$.

\begin{figure}
\begin{center}
\includegraphics[width=.40\linewidth]{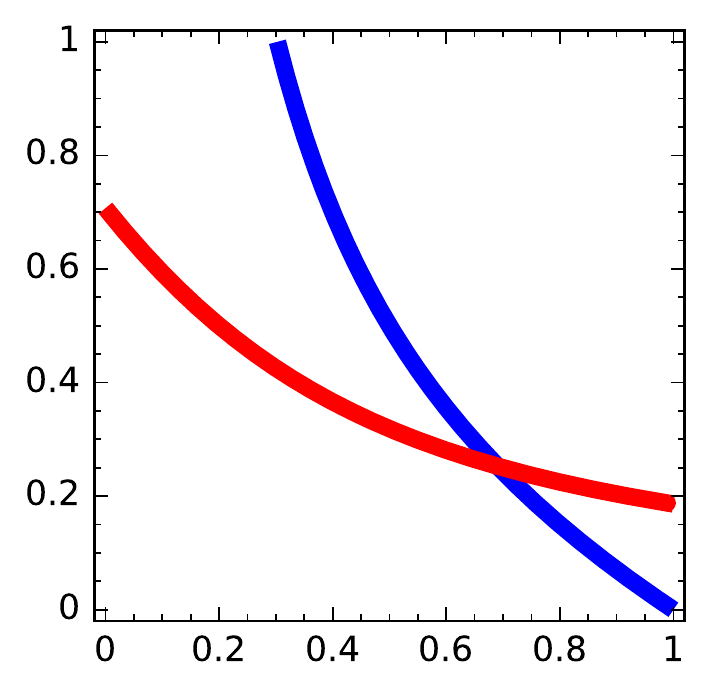}
\includegraphics[width=.45\linewidth,trim={3cm 3cm 3cm 3cm},clip]{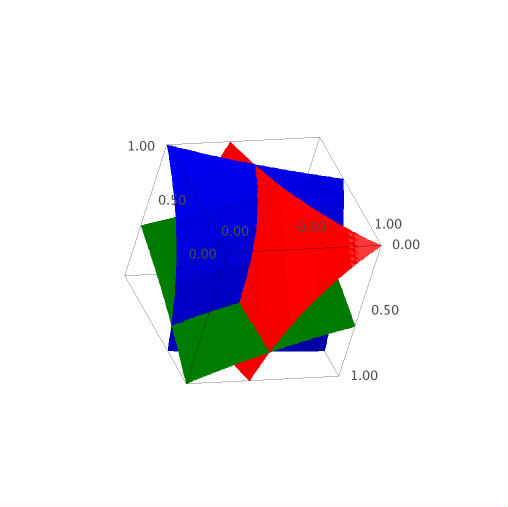}
\end{center}
\caption{
Left:
the two quadratic curves
$x_1^{2} + 3x_1 x_2 = 1$,
$5x_1 x_2 + 2x_2^{2} = 1$,
intersect in a unique point in the box $[0,1]^2$.
Right:
the three quadratic surfaces
$x_1^{2} + 2  x_1 x_2 + 2 x_1 x_3 = 1$,
$x_1 x_2 + x_2^{2} + x_2 x_3 = 1$,
$x_1 x_3 + 3 x_2 x_3 + x_3^{2} = 1$
intersect in a unique point in the box $[0,1]^3$.
}
\label{fig:n23}
\end{figure}

\section{Uniqueness for primitive matrices}

A \emph{primitive} matrix is a nonnegative matrix some power of which is
positive.

\begin{lemma}\label{lem:M2v}
Let $M\in\R_{\geq0}^{n\times n}$ be a primitive matrix
and $v\in\R_{>0}^n$.
If $M^2\,v=v$, then $Mv=v$.
\end{lemma}

\begin{proof}
We already have that $v$ is a positive vector fixed by $M^2$ which is
primitive. But so is $Mv$:
\[
M^2(Mv) = M(M^2v) = Mv.
\]
By Perron's theorem, $v$ and $Mv$ must be colinear, that is, there exists $\lambda\in\R$ such that $v = \lambda Mv$. Then, $v = \lambda^2 M^2v = \lambda^2 v$ and thus $\lambda^2=1$. Since $v$ and $Mv$ are positive, we deduce $\lambda=1$.
\end{proof}

\begin{proposition}\label{prop:perron-unicity}
Let $M\in\R^{n\times n}$ be a primitive matrix.
If there exists a positive vector $\bx=(x_1,\dots,x_n)^t\in\R_{>0}^n$ such that
$M\bx=(\frac{1}{x_1},\dots,\frac{1}{x_n})^t$, then it is unique.
\end{proposition}

\begin{proof}
Let $\bx=(x_1,\dots,x_n)^t,\by=(y_1,\dots,y_n)^t\in\R_{>0}^n$.
Suppose that $X=\diag(\bx)$ and $Y=\diag(\by)$ are such that $XMX$
and $YMY$ are both stochastic.
The product of diagonal matrices commutes, so we have
\[
(XMY)^2 = XM(YX)MY = XM(XY)MY = (XMX)(YMY).
\]
We conclude that $(XMY)^2$ is stochastic.
From Lemma~\ref{lem:M2v}, we conclude that $XMY$ is stochastic.
Thus we have
\[
XMY(1,\dots,1)^t = (1,\dots,1)^t
\quad
\text{ and }
\quad
YMY(1,\dots,1)^t = (1,\dots,1)^t
\]
and
\[
(x_1^{-1},\dots,x_n^{-1})
=X^{-1}(1,\dots,1)^t
= MY(1,\dots,1)^t
= Y^{-1}(1,\dots,1)^t
=(y_1^{-1},\dots,y_n^{-1}).
\]
Therefore $\bx=\by$.
The conclusion follows from Lemma~\ref{lem:3equivalent}.
\end{proof}

\section{Solutions are fixed points}

It can be seen in Figure~\ref{fig:n23} that the surfaces of each equation in
the positive octant are functions of the form $y=f(x)$ or $z=f(x,y)$. We now
formalize and prove this.

Let $M=(m_{ij})\in\R^{n\times n}$ be a nonnegative matrix.
For each $i\in\{1,\dots,n\}$
and $(x_1,\dots,x_{i-1},x_{i+1},\dots,x_n)\in\R_{\geq0}^{n-1}$,
we denote
\begin{equation}\label{eq:bi}
b_i=\sum_{j\neq i} m_{ij}x_j.
\end{equation}
For each $i\in\{1,\dots,n\}$,
we define a function
$f^{(M)}_i:\R_{\geq0}^{n-1}\to\R_{>0}$:
\begin{equation}
f^{(M)}_i(x_1,\dots,x_{i-1},x_{i+1},\dots,x_n) =
\begin{cases}
b_i^{-1} & \text{ if } m_{ii}=0,\\
\displaystyle
\frac{-b_i+\sqrt{b_i^2+4m_{ii}}}{2m_{ii}} & \text{ if } m_{ii}\neq 0.
\end{cases}
\end{equation}

\begin{lemma}\label{lem:existenceoffi}
Let $M=(m_{ij})\in\R^{n\times n}$ be a nonnegative real matrix,
$i\in\{1,\dots,n\}$ and assume $\bx=(x_1,\dots,x_n)^t\in\R_{>0}^n$.
The vector $\bx$ satisfies Equation~\eqref{eq:2} if and only if
$x_i=f_i^{(M)}(x_1,\dots,x_{i-1},x_{i+1},\dots,x_n)$.
\end{lemma}

\begin{proof}
Equation~\eqref{eq:2} can be seen as a quadratic equation of the variable $x_i$:
\begin{align}\label{eq:3}
m_{ii}x_i^2
+ b_ix_i 
- 1 = 0
\end{align}
where $b_i=\sum_{j\neq i} m_{ij}x_j$ is the coefficient of $x_i$ in this quadratic polynomial.
If $m_{ii}=0$, then
$b_ix_i=1$ and there is only one solution
$x_i= b_i^{-1}$ to Equation~\eqref{eq:3}.
Moreover $x_i>0$.
If $m_{ii}\neq 0$, then there are exactly two real solutions
\[
\frac{-b_i-\sqrt{b_i^2+4m_{ii}}}{2m_{ii}}<0
\qquad
\text{and}
\qquad
\frac{-b_i+\sqrt{b_i^2+4m_{ii}}}{2m_{ii}}>0
\]
to Equation~\eqref{eq:3}, the second one being positive.
\end{proof}

For every matrix $M\in\R^{n\times n}$, let
\[
    F^{(M)}:\R_{\geq0}^n\to\R_{>0}^n:\bx\mapsto
\left(\begin{array}{l}
f^{(M)}_1(x_2,x_3,\dots,x_n)\\
f^{(M)}_2(x_1,x_3,\dots,x_n)\\
\dots\\
f^{(M)}_n(x_1,x_2,\dots,x_{n-1})
\end{array}\right)
\]
We now have a new equivalent statement.

\begin{lemma}\label{lem:ifffixedpoint}
Let $M\in\R^{n\times n}$ be a nonnegative matrix and $\bx=(x_1,\dots,x_n)^t$ be a
positive vector. Then
$M\bx=(\frac{1}{x_1},\dots,\frac{1}{x_n})^t$
if and only if
$\bx$ is a fixed point of $F^{(M)}$.
\end{lemma}

\begin{proof}
Let $\bx=(x_1,\dots,x_n)^t>0$.
We have that 
$F^{(M)}(\bx)=\bx$
if and only if
\[
x_i=f^{(M)}_i(x_1,\dots,x_{i-1},x_{i+1},\dots,x_n)
\]
for every $i\in\{1,\dots,n\}$
if and only if $\bx$ satisfies
Equation~\eqref{eq:2}
for every $i\in\{1,\dots,n\}$
from Lemma~\ref{lem:existenceoffi}
if and only if
$M\bx=(\frac{1}{x_1},\dots,\frac{1}{x_n})^t$
from Lemma~\ref{lem:3equivalent}.
\end{proof}

\section{Fixed-Point theorems}

From \cite{MR816732}, we recall some classical fixed-point theorems.


\begin{theorem}[Brouwer Fixed Point Theorem]
Every continuous function from a closed ball of a Euclidean space into itself
has a fixed point.
\end{theorem}

We consider closed balls for the $\infty$-norm. 
For every $\ba=(a_1,\dots,a_n),\bb=(b_1,\dots,b_n)\in\R^n$, 
the closed ball with center $(\ba+\bb)/2$ and radius $max\{0,(a_i-b_i)/2\}$
for every $i$-th coordinate, $1\leq i\leq n$, is denoted by
\[
Box(\ba,\bb) = \{(x_1,\dots,x_n)\in\R^n \mid a_i\leq x_i\leq b_i, 1\leq i\leq n\}.
\]
A function $F:\R_{\geq0}^n\to\R_{\geq0}^n$ is said \emph{decreasing} if
$F(\bx+\bt)\leq F(\bx)$ for every $\bx,\bt\in\R_{\geq0}^n$.

\begin{corollary}
    \label{cor:existence}
Let $F:\R_{\geq0}^n\to\R_{\geq0}^n$ be a continuous and decreasing function.
Then there exists a vector $\bx\in\R_{\geq 0}^n$ such that $\bx=F(\bx)$.
\end{corollary}

\begin{proof}
Since $F$ is continuous and decreasing, we have that, for every
$\ba,\bb\in\R_{\geq0}^n$,
\[
F(Box(\ba,\bb)) \subseteq Box(F(\bb),F(\ba)).
\]
Moreover, $F$ reaches its maximal value at $\ba=\bzero$ so that
$F(\R_{\geq0}^n) \subseteq Box(\bzero, F(\bzero))$.
Then 
\[
F(Box(\bzero,F(\bzero))) \subseteq Box(F^2(\bzero),F(\bzero)) \subseteq Box(\bzero,F(\bzero))
\]
and Brouwer fixed point theorem applies since $Box(\bzero,F(\bzero))$ is a closed ball.
\end{proof}

\subsection{Banach Fixed-Point Theorem}

Let $(X, d)$ be a metric space. Then a map $T : X \to X$ is called a
\emph{contraction} mapping on $X$ if there exists $q\in[0, 1)$ such that
\[
d(T(x),T(y))\leq qd(x,y)
\]
for all $x, y \in X$.

\begin{theorem}[Banach Fixed Point Theorem] 
Let $(X, d)$ be a non-empty complete metric space with a contraction mapping $T
: X \to X$. Then $T$ admits a unique fixed-point $\bx$ in $X$.
\end{theorem}


\section{Existence for nonnegative matrices with positive diagonal entries}

Now we compute the gradient of $f^{(M)}_i$:
\begin{equation}\label{eq:gradientfi}
\vec{\nabla}f^{(M)}_i(\bx) =
\begin{cases}
- b_i^{-2} 
(m_{i1}, \dots, m_{i,i-1}, m_{i,i+1},\dots, m_{in})^t 
& \text{ if } m_{ii}=0,\\
\displaystyle
\frac{1}{2m_{ii}}\left(\frac{b_i}{\sqrt{b_i^2+4m_{ii}}}-1\right)
(m_{i1}, \dots, m_{i,i-1}, m_{i,i+1},\dots, m_{in})^t 
& \text{ if } m_{ii}\neq 0,
\end{cases}
\end{equation}
and we conclude that
$\vec{\nabla}f^{(M)}_i\leq0$.

We now prove the existence of a fixed point of $F^{(M)}$ using Brouwer fixed-point theorem.

\begin{proposition}\label{prop:reductiontobrouwer}
Let $M=(m_{ij})\in\R^{n\times n}$ be a nonnegative real matrix
such that $m_{ii}>0$ for every $i$ with $1\leq i\leq n$.
Then there is a
positive vector $\bx=(x_1,\dots,x_n)^t>0$ such that
$M\bx=(\frac{1}{x_1},\dots,\frac{1}{x_n})^t$.
\end{proposition}

\begin{proof}
The function $F^{(M)}:\R_{\geq0}^n\to\R_{>0}^n$ is continuous since $m_{ii}\neq0$ for all $i$ such that $1\leq i\leq n$.
It is decreasing
since the entries of its gradient are zero or negative.
Thus, Corollary~\ref{cor:existence} applies and
there exists a vector $\bx\in\R_{\geq 0}^n$ such that $\bx=F^{(M)}(\bx)$.
From the definition of $F^{(M)}$, we conclude that the entries of $\bx$ are
positive, i.e., $\bx\in\R_{>0}^n$.
From Lemma~\ref{lem:ifffixedpoint}, we conclude
the existence of a
positive vector $\bx=(x_1,\dots,x_n)^t$ such that
$M\bx=(\frac{1}{x_1},\dots,\frac{1}{x_n})^t$.
\end{proof}

\begin{example}\label{ex:triangular}
Let
\[
        M =
\left(\begin{array}{rrr}
1 & 0 & 0 \\
1 & 1 & 0 \\
1 & 1 & 1
\end{array}\right)
\qquad
\text{ and }
\qquad
\bx =
\left(1,
      \frac{\sqrt{5} - 1}{2},
      \frac{\sqrt{2 \, \sqrt{5} + 22} - \sqrt{5} - 1}{4}\right)^t.
\]
We verify that
\[
M\bx=
\left(1,\frac{\sqrt{5} + 1}{2}, \frac{\sqrt{5} + \sqrt{2 \, \sqrt{5} + 22} + 1}{4}\right)^t
=
\left(1, \frac{2}{\sqrt{5} - 1}, \frac{4}{\sqrt{2 \, \sqrt{5} + 22} - \sqrt{5} - 1}\right)^t.
\]
\end{example}


\begin{proof}[Proof of Theorem~\ref{thm:main}]
Unicity follows from
Proposition~\ref{prop:perron-unicity}
since $M$ is primitive.
Existence follows from
Proposition~\ref{prop:reductiontobrouwer}
since diagonal entries of $M$ are positive.
\end{proof}

Proposition~\ref{prop:reductiontobrouwer} does not include primitive matrices
with zero entries on the diagonal since we can't apply Brouwer fixed-point
theorem when $m_{ii}=0$ for some $i$: $f_i^{(M)}$ is not continuous at $\bzero$
in this case.
But the result still holds (see \cite[Theorem 8.2]{brualdi_diagonal_1966}).
For example, let
\[
M=
\left(\begin{array}{rrr}
    0 & 0 & 1 \\
    1 & 0 & 0 \\
    0 & 1 & 1
\end{array}\right)
\qquad
\text{ and }
\qquad
\bx=\left(\sqrt{2}, \frac{1}{\sqrt{2}}, \frac{1}{\sqrt{2}}\right)^t.
\]
We verify that
$M\bx=\left(\frac{1}{\sqrt{2}},\,\sqrt{2},\,\sqrt{2}\right)^t$.

\section{Uniqueness when diagonal entries are relatively large}

To prove uniqueness in some cases including matrix $M$ from Example~\ref{ex:triangular} which is not primitive, we can use Banach fixed-point theorem.
Note that it is not possible to prove that the map $F^{(M)}$ is a contraction
for every nonnegative matrix $M$. For example, consider
\[
M =
\left(\begin{array}{rrr}
1 & 2m & 2m \\
2m & 1 & 2m \\
2m & 2m & 1
\end{array}\right)
\]
for some $m>0$.
We get that the gradient of
$F^{(M)}$ at $\bx=\bzero$ (in which case $b_i=0$ in Equation~\eqref{eq:gradientfi}) is
\[
\left(
\frac{\partial F^{(M)}}{\partial x_i}(\bzero)
\right)_{i=1,2,3}
=
\left(\begin{array}{rrr}
0 & -m & -m \\
-m & 0 & -m \\
-m & -m & 0
\end{array}\right)
\]
which can get as large as $m$ is.
For some matrices $M$, the map $F^{(M)}$ is a contraction
as we show now.

\begin{proposition}\label{prop:banach-unicity}
Let $M=(m_{ij})\in\R^{n\times n}$ be a nonnegative real matrix
such that $2m_{ii}>m_{ij}$ for every $i,j$ with $1\leq i,j\leq n$.
Then there exists a
unique positive vector $\bx=(x_1,\dots,x_n)^t>0$ such that
$M\bx=(\frac{1}{x_1},\dots,\frac{1}{x_n})^t$.
\end{proposition}

\begin{proof}
Existence follows from Proposition~\ref{prop:reductiontobrouwer}.

To prove uniqueness we use the Banach Fixed Point Theorem and we show that
$F^{(M)}$ is a contraction.
From the hypothesis, there exists a constant $C>0$ such that
\[
0\leq\frac{m_{ij}}{2m_{ii}}\leq C<1
\]
for every $i,j$ with $1\leq i,j\leq n$.
Thus from Equation~\eqref{eq:gradientfi} 
and since
\[
\left\vert\frac{b_i}{\sqrt{b_i^2+4m_{ii}}}-1\right\vert\leq 1
\]
we get that
\begin{equation}
\left\Vert\vec{\nabla}f^{(M)}_i(\bx)\right\Vert_\infty
\leq
\frac{1}{2m_{ii}}
\left\Vert (m_{i1}, \dots, m_{i,i-1}, m_{i,i+1},\dots, m_{in})^t\right\Vert_\infty
\leq C
\end{equation}
for every $\bx\in\R_{\geq0}^{n-1}$ and $1\leq i\leq n$.
The function $f^{(M)}_i$ is differentiable on $\R_{\geq0}^{n-1}$.
Using the Mean value theorem in several variables,
for every $\ba,\bb\in\R_{\geq0}^{n-1}$ there exists $c\in[0,1]$
such that
\[
f^{(M)}_i(\bb)-f^{(M)}_i(\ba)
=
\nabla f^{(M)}_i\left((1-c)\ba+c\bb\right) \cdot (\bb-\ba).
\]
Therefore, by the Cauchy-Schwarz inequality ($|\bx\cdot\by|\leq\Vert\bx\Vert\Vert\by\Vert$),
\[
\displaystyle \left |f^{(M)}_i(\bb)-f^{(M)}_i(\ba)\right|
\leq
C \left\Vert\bb-\ba\right\Vert_\infty.
\]
Thus $f^{(M)}_i$ is a contraction for every $i$ with $1\leq i\leq n$.
Then $F^{(M)}$ is a contraction.
The conclusion is deduced from Lemma~\ref{lem:ifffixedpoint}.
\end{proof}

Proposition~\ref{prop:banach-unicity} seems to hold when $2m_{ii}\leq m_{ij}$. A
possible option in this case is to show that some power of $F^{(M)}$ is a
contraction
and use a stronger version of Banach theorem:
if some iterate $T^n$ of $T$ is a contraction,
then $T$ has a unique fixed point.
More work has to be done.

\section*{Acknowledgements}

I am thankful to
Daniel S. Maynard 
for fruitful exchanges.
The question came up in his research
on interference competition in fungi
\cite{maynard_diversity_2017}
exploring how spatial structure and species diversity
interactively structure the community.

\bibliographystyle{alpha} 
\bibliography{biblio}

\end{document}